\documentclass[12pt]{amsart}
\usepackage{amsfonts,amsmath,amssymb,amsthm,amscd,amsxtra}
\usepackage{enumerate,verbatim}

\usepackage[all,2cell,ps]{xy}
\usepackage[pagebackref]{hyperref}

\theoremstyle{plain} 
\newtheorem{theorem}{Theorem}[section]

\newtheorem{proposition}[theorem]{Proposition}
\newtheorem{corollary}[theorem]{Corollary}
\newtheorem{lemma}[theorem]{Lemma}

\theoremstyle{definition}
\newtheorem{definition}[theorem]{Definition}

\newtheorem{example}[theorem]{Example}

\newtheorem{conj}[theorem]{Conjecture}
\newtheorem{question}[theorem]{Question}

\newtheorem{not&def}[theorem]{Notation and Definitions}

\numberwithin{equation}{section}

\def\hwc{\mathsf{(HWC)}}
\def\arc{\mathsf{(ARC)}}

\newcommand{\fm}{\mathfrak{m}}
\newcommand{\fn}{\mathfrak{n}}
\newcommand{\fp}{\mathfrak{p}}
\newcommand{\fq}{\mathfrak{q}}

\newcommand{\CC}{\mathbb{C}}

\DeclareMathOperator{\Ass}{Ass}

\DeclareMathOperator{\ch}{char}

\DeclareMathOperator{\codim}{codim}

\DeclareMathOperator{\cx}{cx}
\DeclareMathOperator{\depth}{depth}

\DeclareMathOperator{\Ext}{Ext}

\DeclareMathOperator{\height}{height}

\DeclareMathOperator{\Hom}{Hom}

\DeclareMathOperator{\len}{\lambda}

\DeclareMathOperator{\pd}{pd}

\DeclareMathOperator{\Q}{Q}

\DeclareMathOperator{\rank}{rank}

\DeclareMathOperator{\Supp}{Supp}
\DeclareMathOperator{\Spec}{Spec}

\DeclareMathOperator{\Tor}{Tor}

\newcommand{\ul}{\underline}

\newcommand{\tf}[2]{{\boldsymbol\bot}_{#1}{#2}}
\newcommand{\tp}[2]{{\boldsymbol\top}_{\hskip-2pt #1}{#2}}

\def\urltilda{\kern -.15em\lower .7ex\hbox{\~{}}\kern .04em}
\def\urldot{\kern -.10em.\kern -.10em}\def\urlhttp{http\kern -.10em\lower -.1ex
\hbox{:}\kern -.12em\lower 0ex\hbox{/}\kern -.18em\lower 0ex\hbox{/}}

\begin{document}

\title[Vanishing of Tor]{Vanishing of Tor, and why we care about it}

\author{Olgur Celikbas}

\address{Department of Mathematics \\
University of Missouri \\
Columbia, MO 65211, USA}
\email{celikbaso@missouri.edu}

\author{Roger Wiegand}
\address{
Department of Mathematics \\
University of Nebraska--Lincoln \\
Lincoln, NE 68588, USA}
\email{rwiegand1@math.unl.edu}

\thanks{Wiegand's research was partially supported by a Simons Collaboration Grant}



\date{\today}

\begin{abstract}  Given finitely generated modules $M$ and $N$ over a local ring $R$, the tensor product $M\otimes_RN$ typically has nonzero torsion.  Indeed, the assumption that the tensor product is torsion-free influences the structure and vanishing of the modules $\Tor^R_i(M,N)$ for all $i\geq 1$.  In turn, the vanishing of $\Tor^R_i(M,N)$ imposes restrictions on the depth properties of the modules $M$ and $N$.   These connections  made their first appearance in Auslander's 1961 paper ``Modules over unramified regular local rings''.  We will survey the literature on these topics, with emphasis on progress during the past twenty years.\\

This paper is dedicated to Hans-Bjorn Foxby in recognition of his huge impact on
homological algebra, module theory, and commutative algebra.
\end{abstract}

\maketitle

\section{Introduction}
In the first paragraph of his  famous paper {\em Modules over
unramified  regular local rings} \cite{Au}, Auslander wrote:  ``The
main object of study is what it means about two modules $A$ and $B$
over an unramified regular local ring to assert that the torsion
submodule of $A\otimes B$ is zero.''  Indeed, Auslander showed,
assuming that $A$ and $B$ are finitely generated nonzero modules
over such a ring $R$, that both $A$ and $B$ must
be torsion-free, that $\Tor_i^R(A,B)= 0$ for all $i\ge 1$, and that $\pd_RA+\pd_RB =
\pd_R(A\otimes_RB) < \dim R$, unless $R$ is a field.   In this paper we will discuss some of what
has been done on these topics, for more general local rings, in the
50+ years since Auslander's paper.  Our main theme is the same as
Auslander's:  we assume that the tensor product of two nonzero
modules is torsion-free (or is reflexive, or satisfies a certain
 Serre condition, \dots), and then see what we can learn about the
modules $M$ and $N$.  The vanishing of $\Tor_i^R(M,N)$, for all
$i\geq 1$ (or, in some cases, for all $i\gg0$) will play a pivotal role,
as it did in Auslander's paper.

\subsection*{Notation and Assumptions} The notation $(R,\fm,k)$ indicates that $R$ is a
 local (commutative and
Noetherian) ring with maximal ideal $\fm$ and residue field $k$.
The \emph{torsion} submodule of an $R$-module $M$ is the kernel
$\tp{}M$ of the natural map $M\to \Q(R)\otimes_RM$, where $\Q(R) =
\{\text{non-zerodivisors}\}^{-1}R$, the total quotient ring of $R$.
The module $M$ is \emph{torsion} provided $\tp{}M = M$ and
\emph{torsion-free} provided $\tp{}M = 0$.  The torsion-free module
$M/\tp{}M$ is denoted $\tf{}M$.

\section{Rigidity}  We say that the pair $(M,N)$ of finitely generated
$R$-modules is \emph{rigid} provided the
vanishing of $\Tor_j^R(M,N)$ for some $j\ge 1$ forces $\Tor_i^R(M,N)
= 0$ for all $i\ge j$;  and $M$ is \emph{rigid} provided the pair
$(M,N)$ is rigid for every finitely generated $N$.  In the following theorem, we do not
assume that the ring is a regular local ring, but the argument is
essentially the same as Auslander's proof of \cite[Lemma 3.1]{Au}.

\begin{theorem}[Auslander, 1961]\label{thm:rigid-tf}Let $M$ and $N$
be nonzero finitely
generated modules over a reduced local ring $R$. Assume that either
$\tf{}M$ or $\tf{}N$ is  rigid and that $M\otimes_RN$ is
torsion-free.  Then:
\begin{enumerate}[(i)]
\item Both $M$ and $N$ are torsion-free.
\item $\Tor_i^R(M,N)=0$ for all $i\ge 1$.
\end{enumerate}
\end{theorem}

\noindent We'll state the first step of the proof as a lemma, since it will come up several times in this survey.

\begin{lemma}\label{lem:tf-reduc} Let $M$ and $N$ be finitely
generated modules over a local ring $R$, and suppose that $M\otimes_RN$ is torsion-free.  Then the natural maps in the diagram
\begin{equation}\notag{}
\begin{CD}\label{CD:tf-reduc}
M\otimes_RN &\quad \longrightarrow &\quad \tf{}M\otimes_RN\\
\downarrow  && \downarrow \\
M\otimes_R\tf{}N  &\quad  \longrightarrow &\quad \tf{}M\otimes_R\tf{}N
\end{CD}
\end{equation}
are all isomorphisms.  In particular, all four modules are torsion-free.
\end{lemma}

\begin{proof}Tensoring the short exact sequence
\begin{equation}\label{eq:tor-tf}
0 \to \tp{}M \to M \to \tf{}M \to 0
\end{equation}
with $N$, we obtain an exact sequence
$$
\tp{}M\otimes_RN \overset{\alpha}{\to} M\otimes_RN \overset{\beta}{\to} \tf{}M\otimes_RN \to 0 \,.
$$
Since  $\tp{}M\otimes_RN$ is torsion and $M\otimes_RN$ is torsion-free, $\alpha$ must  be the zero-map, and hence $\beta$ (the top arrow in the commutative diagram) is an isomorphism.  By symmetry, so is the left-hand arrow.  Now we know that $M\otimes_R\tf{}N$ is torsion-free, and, on tensoring \eqref{eq:tor-tf} with $\tf{}N$, we see by the same argument that the bottom arrow is an isomorphism.  The right-hand arrow is an isomorphism because the diagram commutes.
\end{proof}
\begin{proof}[Proof of Theorem~\ref{thm:rigid-tf}]
Using the fact that $\Q(R)$ is a direct product of fields, we embed $\tf{}M$ and $\tf{}N$ into
finitely generated  free modules $F$ and $G$, respectively, and get short exact sequences:
\begin{equation}\label{eq:M-in-free}
0 \to \tf{}M \to F \to U \to 0
\end{equation}
\begin{equation}\label{eq:N-in-free}
0 \to \tf{}N \to F \to V \to 0
\end{equation}

Supposing that $\tf{}M$ is rigid, we apply $\tf{}M\otimes_R-$ to \eqref{eq:N-in-free} to obtain an injection $\Tor_1^R(\tf{}M,V)\hookrightarrow \tf{}M\otimes_R\tf{}N$.  Since $\Tor_1^R(\tf{}M,V)$ is torsion (again, because $\Q(R)$ is a direct product of fields) and $ \tf{}M\otimes_R\tf{}N$ is torsion-free, $\Tor_1^R(\tf{}M,V)$ must be $0$.  Therefore $\Tor_2^R(\tf{}M,V)=0$, whence $\Tor_1^R(\tf{}M,\tf{}N)=0$. If, on the other hand, $\tf{}N$ is rigid, we apply $\underline{\ \ }\otimes_RN$ to \eqref{eq:M-in-free} and reach the same conclusion.

Now, tensoring \eqref{eq:tor-tf} with $\tf{}N$, we get an injection
$\tp{}M\otimes_R\tf{}N\hookrightarrow M\otimes_R\tf{}N$.  Therefore $\tp{}M\otimes_R\tf{}N = 0$.  But $\tf{}N\ne 0$, else $M\otimes_RN$ would be a nonzero torsion module.  Therefore $\tp{}M$ must be $0$, that is, $M$ is torsion-free, and by symmetry $N$ is torsion-free as well.  Now $\Tor_1^R(M,N) = 0$, and either $M$ or $N$ is rigid, whence $\Tor_i^R(M,N) = 0$ for all $i \ge 1$.
\end{proof}

A regular local ring $(S,\fn,k)$ is said to be \emph{unramified} provided either $S$ is equi-characteristic (i.e., contains a field) or else $\ch S=0$, $\ch k=p$ and $p \notin \fn^2$.
Here is Auslander's famous ``Rigidity Theorem'' \cite[Corollary 2.2]{Au}:

\begin{theorem}[Auslander, 1961]\label{thm:unram-rigid} Let $M$ be a finitely generated torsion-free module over an unramified regular local ring.  Then $M$ is rigid.
\end{theorem}

In 1966 Lichtenbaum removed the requirements that the regular local ring be unramified and that the module be torsion-free. He showed  \cite[Theorem 3]{Li}:

\begin{theorem}[Lichtenbaum, 1966]\label{thm:Li-i>>0} Let $(S,\fn)$ be an unramified regular local ring, let $f$ a nonzero element of $\fn$, and put $R = S/(f)$.  Let $M$ and $N$ be finitely generated $R$-modules such that $\Tor_i^R(M,N) = 0$ for all $i\gg 0$.  Then the pair $(M, N)$ is rigid.
\end{theorem}

Since the completion of every regular local ring can be written as $S/(f)$, where $(S,\fn)$ is an unramified regular local ring and $f$ is a nonzero element of $\fn$, we have \cite[Corollary 1]{Li}:

\begin{corollary}\label{cor:ram-rigid} Every finitely generated module over a regular local ring is rigid.
\end{corollary}

\begin{corollary}[Auslander and Lichtenbaum]\label{cor:Aus-van} Let $M$ and $N$ be nonzero finitely generated modules over a regular local ring.  If $M\otimes_RN$ is torsion-free, then $\Tor_i^R(M,N) = 0$ for all $i \ge 1$, and both $M$ and $N$ are torsion-free.
\end{corollary}

Much of our focus in this survey is on hypersurfaces, and, more generally, complete intersections. Here are the relevant definitions:

\begin{definition} The \emph{codimension} $\codim(R)$ of a local ring $(R,\fm)$ is the difference $\dim R - \nu_R(\fm)$, where $\nu$ denotes the minimal number of generators required.  A local ring $(R,\fm)$ is a \emph{complete intersection}  provided its $\fm$-adic completion $\widehat R$ is isomorphic to $S/(\underline f)$, where $(S,\fn)$ is a complete regular local ring and $\underline f = (f_1,\dots,f_c)$ is a regular sequence in $\fn$.  The integer $c$ is the \emph{relative codimension} of $\widehat R$ \emph{in} $S$.  We always have  $\codim R \le c$, with equality if and only if $(\underline f)\subseteq \fn^2$.  A \emph{hypersurface} is a complete intersection whose completion has relative codimension one in a complete regular local ring.
\end{definition}

\begin{example}\label{eg:node} Rigidity can fail over hypersurfaces.  Let $R = k[\![x,y]\!]/(xy)$, where $k$ is a field.  For the modules  $M:=R/(x)$  and $N := R/(x^2)$, one checks that
\begin{equation}\notag{}
\Tor^R_i(M,M) \cong \Tor^R_i(M,N) \cong  \begin{cases}{k \text{ if $i$ is odd and positive}}\\ {0  \text{ if $i$ is even and positive}}\end{cases}\,.
\end{equation}
Notice that the coset $x+(x^2)\in N$ is killed by the non-zerodivisor $x+y$.  Thus $\tp{}N \ne 0$, even though $M\otimes_RN \cong M$, which \emph{is} torsion-free.  So, even for one-dimensional hypersurfaces, one can have nonzero modules, one of which has torsion, such that the tensor product is torsion-free.
\end{example}

This example exploits the fact that $R$ is not a domain and, of course, the fact that $M$ has infinite projective dimension.  Here is an example over a one-dimensional domain \cite[Example 4.8]{HW1}:

\begin{example}[Huneke and Wiegand, 1994]\label{eg:twisted} Let $R = k[\![t^3,t^4,t^5]\!]$, with canonical ideal $\omega_R = (t^3,t^4)$.  There exists a finitely generated $R$-module $M$ with nonzero torsion such that $M\otimes_R\omega_R$ is reflexive.
\end{example}

We know of no example over a hypersurface, or even over a complete intersection, of two modules, \emph{both} with nonzero torsion, whose tensor product is torsion-free.

\begin{question}\label{q:one-tf} Let $R$ be a  complete intersection, and let $M$ and $N$ be nonzero $R$-modules such that $M\otimes_RN$ is torsion-free.  Must \emph{at least one} of $M$, $N$ be torsion-free?  What if $R$ is a hypersurface?  What if $R$ is a domain?
What if $\dim R = 1$?
\end{question}

Once we get away from complete intersections, things can go awry in a hurry,
even for one-dimensional  domains \cite{Constapel}:

\begin{example}[Constapel 1996]\label{eg:Constapel} Let $R:=k[\![t^8,\dots,t^{14}]\!]$.  This ring is Gorenstein since the semigroup of exponents is symmetric \cite{Ku70}.  There exist finitely generated $R$-modules $M$ and $N$, both with nonzero torsion,  such that $M\otimes_RN$ is torsion-free.
\end{example}

Given a positive integer $e$, we say the pair $(M,N)$ of finitely generated modules over a local ring $R$ is $e$-\emph{rigid}, provided the vanishing of $\Tor_i^R(M,N)$, for $i = j +1,\dots,  j+e$ (with $j\ge 0$), forces $\Tor_i^R(M,N) = 0$ for all $i>j$.  The module $M$ is $e$-\emph{rigid} provided $(M, N)$ is $e$-rigid for all finitely generated $R$-modules $N$.  Even before Lichtenbaum's results, Murthy \cite{Mu} proved $(c+1)$-rigidity  over any complete intersection of relative codimension $c$ in an unramified regular local ring. Perhaps in anticipation of Lichtenbaum's theorem, Murthy stated his result as follows \cite[Theorem 1.6]{Mu}:

\begin{theorem}[Murthy, 1963]\label{thm:murthy-c+1}Let $(S,\fn)$ be a local domain for which every torsion-free module is rigid, and let $R=S/(\underline f)$, where $\underline f = f_1,\dots,f_c$ is a regular sequence in $\fn$. Suppose $j\ge 1$ and $\Tor_i^R(M,N)=0$ for $i = j, j+1, \dots, j+c$.  Then $\Tor_i^R(M,N) = 0$ for all $i \ge j$.
\end{theorem}

Passing (harmlessly) to the completion and combining Theorem~\ref{thm:murthy-c+1} and Corollary~\ref{cor:ram-rigid}, we have

\begin{theorem}\label{thm:c+1-rigid} If $R$ is a complete intersection of
 codimension $c$, then every finitely
 generated $R$-module is $(c+1)$-rigid.
\end{theorem}

 This can fail for Gorenstein rings that are not complete intersections \cite[Example 3.3]{Jo3}:

\begin{example}[Jorgensen, 2001]\label{eg:codim-3} Let $R=S/I$ where $S$ is the polynomial ring $\CC[\![x_1, x_2, x_3, x_4 , x_5, x_6, x_7, x_8, x_9, x_{10}]\!]$ and $I$ is the ideal of $S$ generated by the five polynomials $f_1=x_3x_5-x_2x_6+x_1x_8$, $f_2=x_4x_5-x_2x_7+x_1x_9$, $f_3=x_4x_6-x_3x_7+x_1x_{10}$, $f_4=x_4x_8-x_3x_9+x_2x_{10}$, and $f_5=x_7x_8-x_6x_9+x_5x_{10}$. Then $R$ is a Gorenstein ring (not a complete intersection) of codimension three. Furthermore, setting $M=R/(x_1, x_2, x_4)$ and $N=R/(x_6, x_8, x_{10})$, we have
$$
\Tor^R_1(M,N)=\Tor^R_2(M,N)=\Tor^R_3(M,N)=\Tor^R_4(M,N)=0,
$$
but $\Tor^R_5(M,N)\ne 0$.
\end{example}

To prove Theorem~\ref{thm:murthy-c+1}, Murthy built  a long exact sequence associated with a
change of rings $S\to R$, where $(S,\fn)$ is a local ring,
$R=S/(f)$, and $f$ is a non-zerodivisor in $\fn$. For finitely generated $R$-modules
$M$ and $N$, we denote $\Tor_i^A(M,N)$ by $T^A_i$, where $A$ is
either $R$ or $S$.  Here  is the long exact sequence:

\begin{equation}\label{eq:les-cor}
\begin{matrix}
\vdots & & \vdots & & \vdots & & \\
T^R_j & \to & T^S_{j+1} & \to & T^R_{j+1} & \to \\
&&&&&&&&\\
T^R_{j-1} & \to & T^S_j & \to & T^R_j & \to \\
\vdots & & \vdots & & \vdots & & \\
T^R_1 & \to & T^S_2 & \to & T_2^R &\to \\
&&&&&&&&\\
T_0^R & \to & T_1^S & \to & T_1^R & \to & 0 \,.
\end{matrix}
\end{equation}

\medskip

The proof of Theorem~\ref{thm:c+1-rigid} now goes by induction on the
codimension.  The base case $c=0$ is  Corollary~\ref{cor:ram-rigid}.
Assuming $c\ge 1$, we pass to the completion and assume that $R =
S/(f)$, where $(S,\fn)$ is a complete intersection of codimension
$c-1$ and $f$ is a non-zerodivisor in $\fn$.  Assuming that
$T_i^R=0$ for $i = j,\dots, j+c$, we see from \eqref{eq:les-cor} that
$T^S_i=0$ for $i = j+1,\dots,j+c$.  By the induction hypothesis, we
have $T^S_{j+c+1}=0$.  This, together with the vanishing of
$T^R_{j+c-1}$, forces $T^R_{j+c+1} = 0$. \qed

\medskip

Lichtenbaum \cite{Li} used the same long exact sequence, obtaining it via a spectral sequence argument.  (Cf. \cite[Lemma 2.1]{HW1}.)

\medskip

Until the 90's, there were no known examples of complete intersections without $2$-rigidity.
But in fact Murthy's result is sharp, as shown in \cite[Example 4.1]{Jo1}:

\begin{example}[Avramov and Jorgensen, 1999]\label{eg:Av-Jo} Let $k$ be any field, and fix a positive integer $c$.  The ring  $R :=
k[\![x_1,\dots,x_c, y_1,\dots,y_c]\!]/(x_1y_1,\dots,x_cy_c)$  is a complete intersection
with $\dim R = \codim R = c$.  Put $N = R/(y_1,\dots,y_c)$.  Given any
non-negative integer $s$, there is a finitely generated $R$-module $M_s$ such that
$\Tor_i^R(M_s,N) = 0$ for $i = s+1,\dots , s+c$, but $\Tor^R_{s+c+1}(M_s,N) \cong k$.
\end{example}

On the other hand, Jorgensen  proved sort of an  ``asymptotic'' $2$-rigidity theorem (\cite[Theorem 3.1]{Jo1}), showing that the vanishing of two consecutive $\Tor$s \emph{sufficiently far out} forces the vanishing of all subsequent $\Tor$s:

\begin{theorem}[Jorgensen, 1999]\label{thm:asymptotic2-rigid} Let $(R,\fm)$ be a $d$-dimensional complete intersection, let $M$ and $N$ be finitely generated $R$-modules such that $M\otimes_RN$ has finite length.   There is an integer $H$ with the following property:  If
$\Tor_i^R(M,N) = 0$ for some even $i > H$ and also for some odd $i > H$, then  $\Tor_n^R(M,N) = 0$ for all
$n> d-\max\{\depth M, \depth N\}$.
\end{theorem}

The integer $H$ in the theorem depends on the modules $M$ and $N$.  One can ask whether there is such an integer $H$ depending only on the ring $R$.  Example~\ref{eg:Av-Jo} does not rule this out, since $M_s\otimes_RN$ has infinite length.  Later, in Sections \ref{sec:more-rigidity} and \ref{sec:eta}, we will see that one can sometimes improve $c+1$-rigidity to $c$-rigidity, for modules $M$ and $N$ over a complete intersection $R$ of codimension $c$, provided $M\otimes_RN$ has finite length.

\medskip

Here is a rigidity theorem, a special case of  \cite[Theorem 3.2]{Bergh08}, where the degrees of the hypothesized vanishing Tors need not be consecutive.

\begin{theorem}[Bergh, 2008]\label{thm:Bergh-Rigid} Let $(R,\fm)$ be a complete intersection of dimension $d$ and codimension $c$, and let $M$ and $N$ be finitely generated $R$-modules.  Suppose there exist an integer $n > p:=d-\depth M$ and an odd positive integer $q$ such that $\Tor ^R_i(M,N) = 0$ for $i = n,n+q,\dots,n+cq$.  Then $\Tor^R_i(M,N) = 0$ for all $i > p$.
\end{theorem}

\section{Modules of finite projective dimension}
Let $R$ be an arbitrary local ring.  In \cite[Theorem 4.3]{Au}, Auslander proved that if $x$ is a zerodivisor in $R$, then $x$ is a zerodivisor on every nonzero rigid $R$-module.  He raised, implicitly, the question of whether every zerodivisor of $R$ must be a zerodivisor on every finitely generated nonzero $R$-module of finite projective dimension.  This became known as Auslander's ``Zerodivisor Conjecture'', and the related question, whether every finitely generated module of finite projective dimension is rigid, became the ``Rigidity Conjecture''.   Thus the Rigidity Conjecture implies the Zerodivisor Conjecture (cf. \cite[page 8]{Ho-Neb}).  In fact, the Zerodivisor Conjecture is also a consequence of the Intersection Conjecture of Peskine and Szpiro \cite[Chapter II, Theorem 2.1]{PS1974}: If $M\ne 0$ and $M\otimes_RN$ has finite length, then $\dim N \le \pd_RM$.  Peskine and Szpiro observed that the Intersection Conjecture would follow from the Rigidity Conjecture. They proved the Intersection Conjecture  in characteristic $p>0$ and for local rings essentially of finite type over a field of characteristic $0$.   In 1987 Roberts \cite{Ro1987} proved it in general.  Thus the Zerodivisor Conjecture is now a theorem.

The Rigidity Conjecture did not fare so well.  In 1993 Heitmann \cite{He-Rigid} constructed a ring $R$ and an $R$-module $M$ with constant rank $2$ and with projective dimension $2$ such that $M$ is not rigid.  This example is minimal in the sense that every module of projective dimension one is rigid (trivially), as is every \emph{torsion} module of projective dimension two \cite[Chap. II, Proposition 1.4]{PS1974}:

\begin{theorem}[Peskine and Szpiro, 1974]\label{thm:pd2-rigid} Let $(R,\fm)$ be a local ring, and let $M$ be a finitely generated torsion $R$-module with projective dimension at most $2$.  Then $M$ is rigid.
\end{theorem}
\noindent The proof is kind of fun, and we'll give it below, after a brief discussion of Euler characteristics.

Heitmann's example is not a complete intersection, and the companion module $N$ that demonstrates failure of rigidity has infinite projective dimension.

\begin{question}\label{q:CI-both-fpd} Let $R$ be a local ring and $M$ and $N$ finitely generated $R$-modules
with $\pd_RM<\infty$.
\begin{enumerate}[(i)]
\item If $R$ is a complete intersection, is the pair $(M,N)$ rigid?
\item If both $M$ and $N$ have finite projective dimension, is $(M,N)$ rigid?
\end{enumerate}
\end{question}

\noindent Bergh \cite{Bergh} observed that an answer to (ii)  implies an affirmative answer to (i):

\begin{proposition} \label{prop:Bergh} Let $R$ be a complete intersection.
Then the following conditions are equivalent:
\begin{enumerate}[(i)]
\item If $M$ is a finitely generated $R$-module of finite projective dimension, then $M$ is rigid.
\item If $M$ and $N$ are finitely generated $R$-modules, both with finite projective dimension, then the pair $(M, N)$ is rigid.
\item If $M$ and $N$ are finitely generated $R$-modules such that $\Tor^{R}_{i}(M,N)=0$ for all $i\gg 0$, then the pair $(M, N)$ is rigid.
\end{enumerate}
\end{proposition}

\noindent This is an immediate consequence of the following theorem  \cite[Theorem 3.6]{Bergh}:

\begin{theorem}[Bergh, 2007]\label{thm:Bergh} Let $R$ be a  complete intersection, and let $M$ and $N$ be finitely generated $R$-modules such that $\Tor^{R}_{i}(M,N)=0$ for all $i\gg 0$. There exist finitely generated $R$-modules $M'$ and $N'$ such that
\begin{enumerate}[(i)]
\item both $M'$ and $N'$ have finite projective dimension;
\item $\depth M' = \depth M$, and $\depth  N' = \depth N$; and
\item $\Tor^{R}_{i}(M',N')\cong \Tor^{R}_{i}(M,N)$ for all $i\geq 1$.
\end{enumerate}
\end{theorem}

\subsection*{The Euler characteristic}  Let $(S,\fn,k)$ be a $d$-dimensional regular local ring, and let $M$ and $N$ be finitely generated $S$-modules such that $M\otimes_SN$ has finite length. One defines the \emph{Euler characteristic} by the formula
\begin{equation}\notag{}
\chi^S(M,N) = \sum_{i=0}^d(-1)^i\len(T_i^S)\,,
\end{equation}
where $\len$ denotes length and $T_i^S=\Tor_i^S(M,N)$.  More generally, there are the \emph{higher} Euler characteristics
\begin{equation}\notag{}
\chi_j^S(M,N) = \sum_{i = j}^d(-1)^{i-j}\len(T_i^S)\,.
\end{equation}
\begin{theorem}\label{thm:chi-quartz} Let $M$ and $N$ be finitely generated modules over a regular local ring $S$.  Assume that $M\otimes_SN$ has finite length.
\begin{enumerate}[(i)]
\item\label{item:chi-van} If $\dim M + \dim N < \dim S$, then $\chi^S(M,N) = 0$.
\item\label{item:unram-chi-van}  Assume $S$ is unramified.  For $j\ge 1$, one has  $\chi^S_j \ge 0$,
with equality if and only if $\Tor_j^S(M,N) = 0$.
\end{enumerate}
\end{theorem}
Part \eqref{item:chi-van} is due, independently, to Roberts \cite{Ro1985} and Gillet-Soul\'e \cite{G-S}.  The inequality in part (2) is due to Lichtenbaum \cite{Li}.  The statement on equality in \eqref{item:unram-chi-van} was proved by Lichtenbaum \cite{Li} for $j\ge 2$ and by Hochster \cite{Ho1984} for $j=1$ .  (The ``if'' direction is clear from rigidity; the ``only if'' direction is the main point.)

The special case of Theorem~\ref{thm:chi-quartz}\eqref{item:chi-van} where $N$ has finite length is easy.  In this case, we have $\dim M < \dim S$, which means that $M$ is a torsion module.  Under these conditions, regularity of the ring is not necessary, as long as $M$ has finite projective dimension.

\begin{proposition}\label{prop:fin-len-chi}Let $(R,\fm,k)$ be a local ring, let $M$ be a finitely generated torsion $R$-module with finite projective dimension, and let $N$ be an $R$-module of finite length.  Then $\chi^R(M,N) = 0$.
\end{proposition}
\begin{proof}Euler characteristics are additive on short exact sequences.  Therefore, by induction on $\len_RN$, it will suffice to verify that $\chi^R(M,k) = 0$.
The $k$-dimension ($= \len$) of $\Tor_i^R(M,k)$ is the $i^{\text{th}}$ Betti number of $M$, that is, the rank of the $i^{\text{th}}$ free module in a minimal free resolution\begin{equation}\label{eq:min-free-res}
0\to R^{(b_n)} \to \dots \to R^{(b_0)} \to M \to 0
\end{equation}
of $M$.  When we tensor \eqref{eq:min-free-res} with the total quotient ring $\Q(R)$, the module $M$ disappears, and we count ranks of free modules in the resulting (split) exact sequence,  to get $\sum_{i=0}^d(-1)^ib_i = 0$, as desired.  \end{proof}

Here is a proof of the theorem of Peskine and Szpiro that torsion modules of projective dimension two are rigid.

\begin{proof}[Proof of Theorem~\ref{thm:pd2-rigid}] There is nothing to prove if $\dim R \le 1$.  Therefore we assume that $\dim R \ge 2$ and proceed by induction on $\dim R$.   We assume also that $M\ne 0$.  Let $N$ be a finitely generated $R$-module with $\Tor_1^R(M,N) = 0$.  We will show that $\Tor^R_2(M,N) = 0$. The inductive hypothesis implies that $(\Tor_2^R(M,N))_\fp= 0$ for each prime ideal $\fp \ne \fm$, that is, $\Tor^R_2(M,N)$ has finite length.

Put $L = \text{H}^0_\fm(N)$, the largest finite-length submodule of $N$.  Then
\begin{equation}\label{eq:pos-depth}
\depth N/L > 0\,.
\end{equation}
Applying $M\otimes_R-$ to the short exact sequence
\begin{equation}\label{eq:quotient}
0\to L \to N \to N/L \to 0\,,
\end{equation}
we get an exact sequence
\begin{equation}\label{eq:right-exact}
\Tor_2^R(M,N) \to \Tor_2^R(M,N/L) \to \Tor^R_1(M,L) \to 0\,.
\end{equation}
Therefore
\begin{equation}\label{eq:tor2-fin-len}
\Tor^R_2(M,N/L) \text{ has finite length}.
\end{equation}

Letting $Z$ be the first syzygy of $M$, we obtain short exact sequences:
\begin{equation}\label{eq:right-res}
0\to Z \to R^{(a)} \to M \to 0
\end{equation}
\begin{equation}\label{eq:left-res}
0\to R^{(c)} \to R^{(b)} \to Z \to 0
\end{equation}
Tensoring these exact sequences with $N/L$, we get an isomorphism and an injection:
\begin{equation}\notag{}
\Tor_2^R(M,N/L) \cong \Tor_1^R(Z,N/L) \qquad\qquad \Tor_1^R(Z,N/L)\hookrightarrow (N/L)^{(c)}
\end{equation}
In view of   \eqref{eq:pos-depth} and \eqref{eq:tor2-fin-len}, these imply that
\begin{equation}\label{eq:tor12-quartz}
\Tor_2^R(M,N/L) = 0\,.
 \end{equation}
 Now $\Tor_1^R(M,L) = 0$ by \eqref{eq:right-exact}, and  Proposition~\ref{prop:fin-len-chi} yields
\begin{equation}\notag{}
\len  \Tor_2^R(M,L) -0 + \len (M\otimes_RL) = \chi^R(M,L) = 0\,.
\end{equation}
Therefore $M\otimes_RL = 0$, and hence   $L=0$.  Now \eqref{eq:tor12-quartz} completes the proof.
\end{proof}

Over a local ring $R$, a finitely generated module of projective dimension $2$ is torsion if and only if its Betti sequence $(b_0,b_1,b_2)$ satisfies the equation $b_1=b_0+b_2$.  Proposition~\ref{prop:fin-len-chi} --- more directly, its proof --- gives the ``only if'' implication; the converse follows from the Auslander-Buchsbaum formula
\begin{equation}\label{eq:AB}
\pd_{R_{\fp}} M_{\fp} + \depth_{R_{\fp}} M_{\fp} = \depth R_{\fp}=0
\end{equation}
for every $\fp \in \Ass R$. Tchernev \cite{Tch} calls a triple $(b_0,b_1,b_2)$ of positive integers \emph{rigid} provided, for every local ring $R$, every $R$-module with Betti sequence $(b_0,b_1,b_2)$ is rigid.  Thus $(a,a+c,c)$ is rigid for every pair of positive integers $a,c$.   He shows also that $(a, c+1,c)$ is rigid.  Heitmann's example
\cite{He-Rigid} has Betti sequence $(8,4,2)$, so this sequence is not rigid.  Tchernev \cite{Tch} gives several families
of non-rigid sequences, for example, $(2a+4,a+2,a)$ for $a\ge 2$.  Both Heitmann \cite{He-Rigid} and Tchernev \cite{Tch}   use the ``universal resolutions'' introduced by Hochster \cite{Ho-Neb} in 1975 and generalized by Bruns \cite{Bruns} in 1984.

\section{More rigidity}\label{sec:more-rigidity}

Over a hypersurface, the long exact sequence \eqref{eq:les-cor}, together with vanishing of the Euler characteristic, yields rigidity theorems.  Here is the ``First Rigidity Theorem'' of Huneke and Wiegand \cite[Theorem 2.4]{HW1}, which, under suitable hypotheses, improves $2$-rigidity (guaranteed by Murthy's Theorem~\ref{thm:murthy-c+1}) to $1$-rigidty:

\begin{theorem}[Huneke and Wiegand, 1994]\label{thm:first-rigidity} Let $(R,\fm)$ be a local ring such that the completion $\widehat R$ is a hypersurface in a regular local ring $S$.  Let $M$ and $N$ be finitely generated $R$-modules such that
\begin{enumerate} [(i)]
\item \label{item:fin-len} $M\otimes_RN$ has finite length, and
\item \label{item:dim-sum} $\dim M + \dim N \le d:=\dim R$.
\end{enumerate}

Assume that $\Tor_j^R(M,N) = 0$ for some $j\ge 1$, and either $S$ is unramified or $j > d$.
Then $\Tor_i^R(M,N) = 0$ for all $i \ge j$.
\end{theorem}

We do not know if the theorem would still be true if assumption \eqref{item:fin-len} were deleted.  There are, however, easy examples to show that assumption \eqref{item:dim-sum} is essential.  Here's one (see \cite[Example 4.1]{HW1}):

\begin{example}\label{eg:3-dim-A1} Let $R = k[\![x,y,u,v]\!]/(uv-xy)$, where $k$ is any field.  Put $I = (x,u)$ and $J = (y,v)$.  Then $(R/I)\otimes_R(R/J)$ has finite length, yet $\Tor_i^R(R/I,R/J) \ne 0$ if and only if $i$ is an even positive integer.  (Notice that $\dim R = 3$, and $\dim R/I + \dim R/J = 4$.)
\end{example}

The proof of Theorem~\ref{thm:first-rigidity} in the unramified case, when $j \le d$, requires a careful analysis of the long exact sequence \eqref{eq:les-cor} and uses non-negativity of the higher Euler characteristics (Theorem~\ref{thm:chi-quartz}\eqref{item:unram-chi-van}).  If $j>d$, the result follows from Theorem~\ref{thm:chi-quartz}\eqref{item:chi-van} and the next proposition, which gives more detailed information.

\begin{proposition}\label{prop:chi-van-rigid} Let $M$ and $N$ be finitely generated modules over a local ring $R$ whose completion $\widehat R$ is a hypersurface in a regular local ring $S$.  Assume that $M\otimes_RN$ has finite length and  that $\chi^S(\widehat M,\widehat N) = 0$.  If $i \ge j > d:=\dim R$, then $\len \Tor_i^R(M,N) = \len \Tor_j^R(M,N)$.
\end{proposition}
\begin{proof} We can harmlessly pass to the completion.  Note that lengths are unchanged by this process, and, moreover, the length of an $\widehat R$-module $X$ is unchanged if $X$ is viewed as an $S$-module.  Thus, we may assume that $R = S/(f)$, where $f$ is a nonzero element of the maximal ideal of $S$.

With the notation of \eqref{eq:les-cor}, we have $T^S_{j+1} = 0$, since $j+1 > d+1 = \dim S$ and $S$ is regular.  Consider the resulting exact sequence
\begin{equation}\label{eq:mes}
0 \to T^R_{j+1} \to T^R_{j-1} \to T^S_j \to T^R_j\to \dots \to T^S_0 \to T_1^S\to T_1^R\to 0\,.
\end{equation}
(We changed the base ring on the tensor product $T_0$ from $R$ to $S$, for convenient counting.)  Put $t_i^A = \len T^A_i$, where $A$ is either $R$ or $S$.  We want to show that  $t^R_{j+1} = t^R_j$.
The alternating sum of the lengths of the modules in \eqref{eq:mes} is zero.  The alternating sum of the $t^S_i$ is the Euler characteristic, which is zero by assumption.  Most of the $t^R_i$ cancel in pairs, leaving only $t^R_{j+1}$ and $t^R_j$, and these occur with opposite signs.  That does it!
\end{proof}

Using the long exact sequence \eqref{eq:les-cor} and induction, one gets the following extension of Theorem~\ref{thm:first-rigidity}  (see \cite[Theorem 1.9]{HJW}):

\begin{theorem}[Huneke, Jorgensen and Wiegand, 2001] \label{thm:HJW} Let $R$ be a $d$-dimensional local ring such that $\widehat{R}=S/(\underline{f})$ where $(S,\mathfrak{n})$ is a complete regular local ring and $\underline{f}=f_{1},f_{2},\dots,f_{c}$ is a regular sequence in $\fn$. Let $M$ and $N$ be finitely generated $R$-modules. Assume the following conditions hold:
\begin{enumerate}[(i)]
    \item $M\otimes_{R}N$ has finite length.
    \item $\dim(M)+\dim(N)<d+c$.
    \item $\Tor^{R}_{j}(M,N)=\ldots =\Tor^{R}_{j+c-1}(M,N)=0$ for some positive integer $j$.
    \item Either $j>d$ or else $S$ is unramified.
\end{enumerate}
Then $\Tor^{R}_{i}(M,N)=0$ for all $i\geq j$.
\end{theorem}

In the next section we will encounter a sharper version of the theorem, due to Jorgensen, in the case $n>d$.  For now, we mention another $c$-rigidity theorem, due to Celikbas \cite[Proposition 4.9]{Ce}, where assumptions (i) and (ii) are replaced by the assumption that all three modules are Cohen-Macaulay.  (Recall that a finitely generated module $M$ over a local ring $R$ is \emph{Cohen-Macaulay} provided $\depth M = \dim M$ and \emph{maximal Cohen-Macaulay} (MCM) provided $\depth M = \dim R$.)

\begin{theorem}[Celikbas, 2010] \label{thm:Ce-CM} Let $R$ be a $d$-dimensional local ring such that $\widehat{R}=S/(\underline{f})$ where $(S,\mathfrak{n})$ is a complete, unramified regular local ring and $\underline{f}=f_{1},f_{2},\dots,f_{c}$ is a regular sequence in $\fn$ with $c\ge 2$.  Let $M$ and $N$ be finitely generated $R$-modules. Assume the following conditions hold:
\begin{enumerate}[(i)]
    \item $M$, $N$, and $M\otimes_{R}N$ are Cohen-Macaulay.
    \item $\Tor^{R}_1(M,N)=\ldots =\Tor^{R}_c(M,N)=0$.
\end{enumerate}
Then $\Tor_i^R(M,N) = 0$ for all $i\ge 1$.
\end{theorem}

The motivation for Theorem \ref{thm:Ce-CM} comes from an unpublished result of Huneke and Wiegand; see \cite[Theorem 4.1]{Ce}. In Theorems \ref{thm:HJW} and \ref{thm:Ce-CM} the assumption, in the cited references, that $(\underline{f})\subseteq \mathfrak n^2$
is unncessary:  If $(\underline{f})\not\subseteq \mathfrak n^2$, then $R$ is a complete
intersection of codimension $c-1$, and Murthy's theorem \ref{thm:murthy-c+1} implies $c$-rigidity.

\section{Complexity}\label{sec:complexity}

One can find several variations and  generalizations of these rigidity results in the literature; see, for example, \cite{Ce}, \cite{Da1}, \cite{Jo1}, \cite{Jo2} and \cite{Sadeghi1}.  Here we mention two, due to Jorgensen \cite[Proposition 2.3 and Theorem 2.6]{Jo1}, that depend on  \emph{complexity}  (see \cite{AvBu}, \cite{Da2}).

Recall that the \textit{complexity} $\cx B$ of a sequence of nonnegative integers $B =
(b_i)_{i\geq 0}$ is defined to be the smallest non-negative integer $r$ (if one exists) such that, for some real number $A$, one has $b_n\leq A \cdot n^{r-1}$ for all $n\gg 0$.
Thus, for example, $\cx B \le 1$ if and only if $B$ is bounded, and $\cx B = 0$ if and only if $b_n= 0$ for all $n\gg0$.  If, now, $R$ is a local ring, the complexity  $\cx_RM$ of a finitely generated $R$-module $M$ is defined by $\cx_RM = \cx(\beta_i^R(M))$, where $\beta_i$ is the $i^{\text{th}}$ Betti number of $M$, that is, the rank of the $i^{\text{th}}$ free module in a minimal resolution of $M$. Over a complete intersection $(R,\mathfrak m)$, one always has $\cx_R(M) \le \codim R$ (see Shamash \cite{Sh} or Gulliksen \cite{Gu}).  Here is \cite[Proposition 2.3]{Jo1}:

\begin{theorem} [Jorgensen, 1999]  \label{thm:Jo-cx} Let $R$ be a  complete intersection, and let $M$ and $N$ be finitely generated $R$-modules. Set $d=\dim(R)$, $r=\min \{ \cx(M), \cx(N) \}$ and $b=\max\{\depth M, \depth N\}$. If $\Tor^R_n(M,N),\dots,\Tor^R_{n+r}(M,N)$ vanish
for some integer $n\geq d-b+1$, then $\Tor^{R}_{i}(M,N)=0$ for all $i\geq d-b+1$.
\end{theorem}

The modules $M_s$ and $N$ in
Example~\ref{eg:Av-Jo} both have complexity $c$ (see  \cite[Example 4.1]{Jo1}); again we see that the vanishing interval (of length $r+1$) cannot in general be shortened.  However, when the tensor product of the modules has finite length and the complexities of the modules are equal, one can do better (see \cite[Theorem 2.6]{Jo1}):

\begin{theorem}[Jorgensen, 1999] \label{thm:Jo-fin-len-rigid}
Let $R$ be a  complete intersection, and let $M$ and $N$ be finitely generated $R$-modules. Set $d=\dim(R)$, $r=\max \{ \cx(M), \cx(N), 1 \}$, and $b=\depth M + \depth N$.   Assume
\begin{enumerate}[(i)]
\item $M\otimes_RN$ has finite length, and
\item $\dim M + \dim N < d + r$.
\end{enumerate}
If $\Tor^R_n(M,N)= \dots=\Tor^R_{n+r-1}(M,N) = 0$
for some integer $n\geq d-b+1$, then $\Tor^{R}_{i}(M,N)=0$ for all $i\geq d-b+1$.
\end{theorem}

\section{Vanishing}\label{sec:vanishing}

 In this section we seek conditions that force $M$ and $N$ (finitely generated modules over a local ring $(R,\fm )$) to be \emph{Tor-independent}, by which we mean that $\Tor_i^R(M,N) = 0$ for all $i\ge 1$.  An important consequence of Tor-independence, at least over complete intersections \cite[Proposition 2.5]{HW1}, is the \emph{depth formula}.

\begin{theorem}[Huneke and Wiegand, 1994]\label{thm:DF}
Let $M$ and $N$ be Tor-independent finitely generated modules
 over a complete intersection $(R,\fm)$.   Then the depth formula holds:
 \begin{equation}\label{eq:DF}
 \depth M + \depth N = \depth (M\otimes_RN) + \depth R\,.
 \end{equation}
 \end{theorem}

In view of the Auslander-Buchsbaum formula \eqref{eq:AB}, the depth formula \eqref{eq:DF} is  a natural extension
of Auslander's formula \cite[Theorem 1.2]{Au}
\begin{equation}\label{eq:aus-DF}
\depth N =\depth (M\otimes_RN) + \pd_RM
\end{equation}
(valid whenever $\pd_RM <\infty$ and $M$ and $N$ are Tor-independent) to modules of infinite projective dimension.
When \emph{both} modules have finite projective dimension and are Tor-independent,
 Auslander's formula takes the pleasant form $\pd_RM +\pd_RN = \pd_R(M\otimes_RN)$, as stated in the introduction.
 The strict inequality $\pd_R(A\otimes_RB) < \dim R$, also mentioned in the introduction, is equivalent, by \eqref{eq:AB}, to the inequality $\depth (A\otimes_RB) >0$, which holds since, by assumption, $A\otimes_RB$ is torsion-free and $\dim R > 0$.

The depth formula holds for Tor-independent modules over an arbitrary local ring provided one of the modules has finite complete intersection dimension \cite{AGP}; see \cite{AY}, \cite{Iy}.   There appear to be no known counterexamples to the depth formula, for Tor-independent finitely generated modules over arbitrary local rings; see also \cite{BeJ}, \cite{Sadeghi2} and \cite{SSY}.

Vanishing of $\Tor_i^R(M,N)$ for all $i\gg0$ is not enough to yield the depth formula.  For example, let $(R,\fm,k)$ be a discrete valuation ring and take $M = N = k$.  There is, however, another version of the depth formula, proposed by Auslander \cite{Au}.  Supposing that $\Tor^R_i(M,N) = 0$ for all $i\gg0$, set $q = \max\{j \mid \Tor_j^R(M,N) \ne 0\}$.  The proposed general depth formula is the following:
\begin{equation}\label{eq:gen-DF}
\depth M + \depth N = \depth R + \depth \Tor_q^R(M,N) - q
\end{equation}
\noindent Here $M$ and $N$ are understood to be finitely generated modules over a local ring $R$.  Auslander \cite[Theorem 1.2]{Au} verified this formula provided $\pd_R M < \infty$ and $\depth \Tor_q^R(M,N) \le 1$.  Suppose now that $M$ has finite complete intersection dimension \cite{AGP} (e.g., either $M$ has finite projective dimension or $R$ is a complete intersection).  Araya and Yoshino \cite{AY} showed  that the formula holds when $\depth \Tor_q^R(M,N) \le 1$; also, they showed that the formula can fail when $\depth \Tor_q^R(M,N) = 2$.  Choi and Iyengar \cite{CI} asked  whether there is always an index $j$ between $0$ and $q$ for which
\begin{equation}\label{eq:gen-gen-DF}
\depth M + \depth N = \depth R + \depth \Tor_j^R(M,N) - j\,,
\end{equation}
and showed that even \eqref{eq:gen-gen-DF} can fail.  In \cite{HW3}, Huneke and Wiegand invented a game whose goal is to find such an index $j$ and determined conditions under which one loses the game.

\medskip

We now return to our general theme:  Good depth properties of the tensor product force vanishing of Tor (and, by the depth formula, yield information on the depths of the modules involved.  We begin by revisiting the ring of Example~\ref{eg:3-dim-A1} (see \cite[Example 4.1]{HW1}).

\begin{example}\label{eg:tf-not-enuf}Let $R = k[\![x,y,u,v]\!]/(uv-xy)$, where $k$ is any field.  Put $I = (x,u)$ and $L = (y,u)$.  Then $I\otimes_RL$ is torsion-free, and $\Tor^R_i(I,L) \ne 0$ if and only if $i$ is an odd positive integer.  Moreover $I$ and $L$ are MCM $R$-modules.
\end{example}

Thus, to get vanishing of Tor over hypersurface domains, it's not enough to assume that the tensor product is torsion-free.  We want it to be \emph{reflexive} (see \cite[Theorem 2.7]{HW1}, stated as Theorem~\ref{thm:2nd-rigidity} below).

\begin{definition}\label{def:rank} Let $M$ be a finitely generated module over a Noetherian ring $R$, and let $X\subseteq \Spec(R)$.  We say that $M$ is \emph{free of constant rank on} $X$ provided  there is an integer $r$ (called the \emph{rank of} $M$ \emph{on} $X$ and denoted $\rank_X M$) such that $M_\fp \cong R_\fp^{(r)}$ for each $\fp\in X$.  We say that $M$ \emph{has rank} provided $M$ is free of constant rank on $\Ass R$, and in this case we let $\rank M
= \rank_{\Ass R} M$. \end{definition}

\begin{theorem}[Huneke and Wiegand, 1994]\label{thm:2nd-rigidity} Let $M$ and $N$ be nonzero finitely generated modules over a hypersurface $R$, and assume $M$ has rank.  If $M\otimes_RN$ is reflexive, then $M$ and $N$ are Tor-independent.  Moreover, $M$ is torsion-free, and $N$ is reflexive.
\end{theorem}

We will sketch the proof of the ``Moreover'' part, since it was a bit overstated in \cite{HW1} (cf. the erratum in Math. Ann.  \textbf{338} (2007)).  It is important to realize that the depth formula has no content if one of the modules is zero (by convention, $\depth 0 = \infty$).  This seemingly trivial point becomes important if we happen to localize the vanishing of Tor at a prime outside the support of one of the modules.  Here we remind the reader of Serre's conditions.

\begin{definition}\label{def:Serre}  Let $M$ be a finitely generated module over a Noetherian ring $R$.  Then $M$ satisfies Serre's condition (S$_n$) provided
\begin{enumerate}[(i)]
\item $M_\fp$ is either MCM or $0$ if $\height \fp \le n$, and
\item $\depth M_\fp \ge n$ if $\height \fp >n$.
\end{enumerate}
\end{definition}

Over complete intersections (in fact, over rings that are Gorenstein in codimension one), a finitely generated module is torsion-free, respectively, reflexive if and only if it satisfies (S$_1$), respectively, (S$_2$).  (See, for example, \cite[Appendix I]{LW}.)

\begin{proof}[Proof of the ``Moreover'' part of Theorem~\ref{thm:2nd-rigidity}] (See also \cite[Remark 1.3]{CP}.)  Obviously $\rank M > 0$, else $M\otimes_RN$ would be a nonzero torsion module, contradicting reflexivity.  Thus the support of $M$ is the whole spectrum.  The depth formula \eqref{eq:DF}, suitably localized, then shows that $N$ satisfies (S$_2$) and hence is reflexive.  To show that $M$ is torsion-free, we note that $\tf{}M\otimes_RN$ is reflexive by Lemma~\ref{lem:tf-reduc}, and hence $\Tor_i^R(\tf{}M,N) = 0$ for all $i\ge0$ by the first part of the theorem.  We tensor the short exact sequence \eqref{eq:tor-tf} with $N$, getting an injection $\tp{}M\otimes_RN \hookrightarrow M\otimes_RN$.  It follows that $(\tp{}M)\otimes_RN = 0$, and hence $\tp{}M = 0$.
\end{proof}

A comparison of  the Auslander-Lichtenbaum theorem (Corollary~\ref{cor:Aus-van}) and Theorem~\ref{thm:2nd-rigidity} suggests the following (somewhat vague) conjecture:

\begin{conj}\label{conj:Sc+1} Let $M$ and $N$ be finitely generated modules over a complete intersection of codimension $c$.  Assume that $M$ and $N$ have ``good'' behavior at primes of low height.  If $M\otimes_RN$ satisfies Serre's condition (S$_{c+1}$), then $\Tor^R_i(M,N) = 0$ for all $i\ge 1$.
\end{conj}

This conjecture was addressed in \cite{HJW}, with some success in codimensions $2$ and $3$.  Here is the result for $c=2$ (\cite[Theorem 2.4]{HJW}):

\begin{theorem}[Huneke, Jorgensen and Wiegand, 2001]\label{thm:HJW-codim2} Let $M$ and $N$ be finitely generated modules over a complete intersection $(R,\fm)$ of codimension $2$.  Assume:
\begin{enumerate}[(i)]
\item $M$ has rank.
\item $M$ is free of constant rank on $\{\fp\in \Spec(R) \mid \height \fp \le 1\}$.
\item $M$ and $N$ satisfy (S$_2$).
\item $M\otimes_RN$ satisfies (S$_3$).
\end{enumerate}
Then $M$ and $N$ are Tor-independent.
\end{theorem}

In \cite{Da2} Dao studied the unramified case, and proved the following vanishing theorem for arbitrary codimensions (\cite[Theorem 7.6]{Da2}):

\begin{theorem}[Dao, 2007 arXiv]\label{thm:Long-CI} Let $(R,\fm)$ be a local ring whose completion $\widehat R$ is a complete intersection of relative codimension $c$ in an unramified regular local ring.  Let $M$ and $N$ be finitely generated $R$-modules, and assume:
\begin{enumerate}[\rm(i)]
\item $M_\fp$ is $R_\fp$-free for each prime ideal $\fp$ of height at most $c$.
\item $M$ and $N$ satisfy (S$_c$).
\item $M\otimes_RN$ satisfies (S$_{c+1}$).
\end{enumerate}
Then $M$ and $N$ are Tor-independent.
\end{theorem}

Theorem \ref{thm:Long-CI} justifies Conjecture \ref{conj:Sc+1} for some noteworthy cases. However, if $\dim R \le c$, Dao's theorem has little content, since condition (i) then implies that $M$ is free.  The following result \cite[Theorem 3.4]{Ce} is of interest even when $\dim R = c$:

\begin{theorem}[Celikbas, 2011]\label{thm:Olgur-rigidity} Let $(R,\fm)$ be a local ring whose completion $\widehat R$ is isomorphic to $S/(\underline f)$, where $(S,\fn)$ is an  unramified regular local ring and $\underline f = f_1,\dots,f_c$ is a regular sequence in $\fn^2$.  Let $M$ and $N$ be finitely generated $R$-modules, and assume:
\begin{enumerate}[\rm(i)]
\item $(\Tor^R_i(M,N))_\fq = 0$ for all $i\ge 1$ and for all prime ideals $\fq$ of height at most $c-1$ (e.g., $M_\fq$ is a free $R_{\fq}$-module for each such prime $\fq$.)
\item $M$ and $N$ satisfy (S$_{c-1}$).
\item $M\otimes_RN$ satisfies (S$_c$).
\end{enumerate}
If $M$ or $N$ has non-maximal complexity (i.e., $\min\{\cx_RM,\cx_RN\} < c$), then $M$ and $N$ are Tor-independent.
\end{theorem}

Example~\ref{eg:node}, with $c=1$, shows that one cannot delete the assumption concerning complexity.

In order to prove Theorem~\ref{thm:Long-CI}, Dao introduced and developed a pairing $\eta^R(-,-)$ defined on pairs of finitely generated modules over a Noetherian ring, particularly over complete intersections.  This pairing, which we shall consider in detail in the next section, generalizes the $\theta$-pairing for hypersurfaces, introduced by Hochster \cite{Ho}  in 1980.

\section{Thetas, and etas, and {T}ors! Oh, my!}\label{sec:eta}

\subsection*{The theta-pairing} Let $(R,\fm)$ be a hypersurface, and let $M$ and $N$ be finitely generated $R$-modules. A consequence of Eisenbud's theory of matrix factorizations \cite{Ei} is that the minimal free resolution of $M$ is eventually periodic, with period at most two \cite[Theorem 6.1]{Ei}.  It follows that $\Tor_i^R(M,N) \cong \Tor_{i+2}^R(M,N)$ for all $i\gg 0$ (see also the long exact sequence of \ref{eq:les-cor}). Now assume that
 $\Tor_{i}^R(M,N)$ has finite length for all $i\gg 0$. (This condition holds for all $M$, $N$ if $(R,\fm)$ is an \emph{isolated singularity}, that is $R_\fp$ is a regular local ring for each prime ideal $\fp \ne \fm$.)  Under these circumstances, Hochster's $\theta$-pairing is defined as
\begin{equation}\notag{}
\theta^{R}(M,N)=\len(\Tor^{R}_{2i}(M,N))-\len(\Tor^{R}_{2i+1}(M,N))
\text{ for  all $i\gg 0$.}
\end{equation}
Hochster \cite{Ho} introduced this pairing in 1981  in order to study the Direct Summand Conjecture.  In \cite[Theorem 2.8]{Da1}, Dao discovered an important connection between the $\theta$-pairing and  rigidity:

\begin{theorem}[Dao, 2006] \label{thm:theta} Let $(R,\fm)$ be a local ring whose completion $\widehat R$ is hypersurface in an unramified regular local ring, and let $M$ and $N$ be finitely generated $R$-modules. Assume  $\Tor_{i}^R(M,N)$ has finite length for all $i\gg 0$. If $\theta^R(M,N)=0$, then $(M,N)$ is rigid.
\end{theorem}

The following example \cite[Example 6.7]{Da1} shows that the converse of Theorem \ref{thm:theta} is not true in general.

\begin{example}[Dao, 2006] Let $R=\CC[\![x,y,u,v]\!]/(xu-yv)$. Let $\fp=(x,y)$, $\fq=(x,v)$, $N=R/\fp \oplus R/\fq$, and $M=R/\fp \oplus N$. Then $M$ is a rigid module, but $\theta^{R}(M,R/\fq)=\theta^{R}(R/\fp,R/\fq)=1$.
\end{example}

Because of Theorem~\ref{thm:theta}, it is important to find classes of rings where the $\theta$-pairing vanishes.   Here are two such classes, from \cite[Theorem 4.1]{Da1}:

\begin{theorem} [Dao, 2006] \label{thm:GR}  Let $(R,\fm)$ be a local ring whose completion $\widehat R$ is hypersurface in an unramified regular local ring, and let $M$ and $N$ be finitely generated $R$-modules. Assume one of the following conditions holds:
\begin{enumerate}[\rm(i)]
\item $R$ is a two-dimensional normal ring.
\item $R$ is a four-dimensional equicharacteristic isolated singularity.
\end{enumerate}
Then $\theta^{R}(M,N)=0$ and hence the pair $(M,N)$ is rigid.
\end{theorem}

In view of this result, Dao \cite[3.15]{Da1} conjectured the following:

\begin{conj}[Dao, 2006]\label{conjtheta} Let $(R,\fm)$ be a local ring whose completion $\widehat R$ is hypersurface in an unramified regular local ring. Assume $R$ is an \textit{even-dimensional} isolated singularity. Then $\theta^{R}(-,-)=0$.
\end{conj}

There have been important recent developments concerning the $\theta$-pairing. In particular, Moore, Piepmeyer, Spiroff and Walker \cite[Theorem 3.2]{MPSW} proved the graded version of Conjecture \ref{conjtheta} using sophisticated techniques of algebraic geometry.

One can show that $\theta^R(-,-)=0$ if $R$ is a one-dimensional hypersurface domain \cite[Proposition 3.3]{Da1}; thus  in this case every finitely generated $R$-module is rigid.   One can propose, more generally:

\begin{conj} \label{conj:rigid-onedim-ci} Let $(R,\fm)$ be a one-dimensional complete intersection domain.  Then all finitely generated $R$-modules are rigid.
\end{conj}

Indeed, this is a very special case of Dao's Conjecture 4.2 in \cite{Da4}, for complete intersections of arbitrary dimension, concerning rigidity of modules that are free of constant rank on the punctured spectrum $\Spec R \setminus \{\fm\}$.
We have already seen that Conjecture \ref{conj:rigid-onedim-ci} can fail for one-dimensional domains that are not complete intersections. (See Examples \ref{eg:twisted} and \ref{eg:Constapel}, and consult Theorem~\ref{thm:rigid-tf}.) Actually, an example was found by Hochster and Huneke in the early 1980's:

\begin{example} Let $R=k[\![t^9,t^{11},t^{13},t^{15},t^{17},t^{19},t^{21}, t^{23}]\!]$, $I=(t^9,t^{11},t^{13},t^{21})$, and $J=(t^{15},t^{17},t^{19}, t^{23})$. Then $R$ is a one-dimensional Gorenstein domain which is not a complete intersection. Furthermore $I+J=\mathfrak{m}=(t^9, t^{11}, t^{13},t^{15}, t^{17},t^{19}, t^{21}, t^{23})$. Notice $J^{2} \subseteq IJ$ so that $I \cap J \subseteq \mathfrak{m} J=J^{2}+IJ \subseteq IJ$, i.e., $I \cap J=IJ$. Therefore, setting $M=R/I$ and $N=R/J$, we see that  $\Tor_1^{R}(M,N)=0$.

Assume now that $\Tor_2^{R}(M,N)=0$. Then $I\otimes_{R}J$ is torsion-free by \cite[Lemma 1.4]{HW1}, and it follows that $I\otimes_{R}J \cong IJ$.  Now $I\otimes_RJ$ needs $16$ generators, by Nakayama's Lemma.  On the other hand, $IJ$ is generated by the nine elements $t^{24}, t^{26}, t^{28}, t^{30}, t^{32}, t^{34}, t^{36}, t^{38}, t^{40}$.  This contradiction shows that  $\Tor_2^{R}(M,N) \neq 0$, and the pair $(M,N)$ is not rigid.
\end{example}

We refer the interested reader to Dao's survey paper \cite{surveydao} for more information on the properties and applications of the $\theta$-pairing.

\subsection*{The eta-pairing} Let $R$ be a commutative Noetherian ring, and let $M$ and $N$ be finitely generated $R$-modules. Assume that there is some non-negative integer $m$ such that $\Tor_{i}^R(M,N)$ has finite length for all $i\geq m$. In \cite{Da2} Dao introduced and studied a generalization of Hochster's $\theta$-pairing:
\begin{equation}\notag{}
\eta_e^R(M,N) = \lim_{n\to\infty}
\frac{1}{n^e}\sum_{i=m}\limits^{n}(-1)^{i} \len_{}(\Tor_i^R(M,N))
\end{equation}
Here $e$ is any positive integer, so that $\eta_{e}^{R}(M,N)$ can be infinite. However, if $(R,\fm,k)$ is a complete intersection of  codimension $c$, then $\eta_{c}^R(M,N)<\infty$, and $\eta_e^R(M,N) = 0$ if  if $e > c$ \cite[Theorem 4.3]{Da2}. Furthermore, if $\widehat R$ is a complete intersection of relative codimension $c$ in an unramified regular local ring, and if $\eta_{c}^R(M,N)=0$, then the pair $(M,N)$ is $c$-rigid.

The $\eta$-pairing is, up to a factor of two, a generalization of Hochster's $\theta$-pairing: over a hypersurface $R$, one has $2\cdot \eta^R_1(M,N) = \theta^R(M,N)$, provided $\Tor^R_i(M,N)$ has finite length for all $i\gg0$; see also \cite{CeD2}. Thus, for example, $\eta^R_1(M,N) = -\frac{1}{2}$ in Example~\ref{eg:node}.

In \cite{GORS} the authors, along with Iyengar and Piepmeyer, exploit the $\eta$ pairing and improve Theorem \ref{thm:Long-CI} under the assumption that $\eta(M,N)=0$:

\begin{theorem}[O. Celikbas, Piepmeyer, Iyengar and R. Wiegand, 2013] \label{GORSthm}
Let $(R,\fm)$ be a local ring whose completion $\widehat R$ is a complete intersection of positive codimension $c$ in an unramified regular local ring.  Let $M$ and $N$ be finitely generated $R$-modules and assume $\dim(R)\geq c$. Assume further that:
\begin{enumerate}[\rm(i)]
\item $M\otimes_{R}N$ satisfies $(S_{c})$.
\item $M$ and $N$ satisfy $(S_{c-1})$.
\item $\len(\Tor^R_i(M,N))<\infty$ for all $i\gg 0$.
\item $\eta_{c}^{R}(M,N)=0$.
\item If $c=1$, assume in addition that $\Supp(\tp{} N) \subseteq \Supp(M)$. (e.g., $N$ is torsion-free.)
\end{enumerate}
Then $\Tor^{R}_{i}(M,N)=0$ for all $i\geq 1$.
\end{theorem}

\noindent A consequence of Theorem~\ref{GORSthm} is that Dao's Theorem~\ref{thm:Long-CI} is true even if the ambient regular local ring is ramified (see \cite[Corollary 5.12]{GORS}).  

The big question:  When is $\eta_c^R(M,N) = 0$?  Since the pair $(M_s, N)$ in Example~\ref{eg:Av-Jo} is not $c$-rigid, we know that $\eta^R_c(M_s,N) \ne 0$.  The problem is that this ring is not an isolated singularity.  Conjecturally (see  \cite{MPSW2}), this is the only obstruction when $c\ge 2$:

\begin{conj}[Moore, Piepmeyer, Spiroff, and Walker]\label{conj:MPSW2} Let $(R,\fm)$ be a local ring whose completion $\widehat R$ is a complete intersection of relative codimension $c$ in an unramified regular local ring.  Assume that $c\ge 2$ and that $R$ is an isolated singularity.  Then $\eta^R_c(M,N) = 0$ for all finitely generated $R$-modules $M$ and $N$.
\end{conj}

The main theorem in the paper by Moore \emph{et al} is \cite[Theorem 4.5]{MPSW2}, which states that the \emph{standard graded} version of this conjecture is true.  (The modules $M$ and $N$ are not required to be graded.)  Polishchuk and Vaintrob~\cite[Remark 4.1.5]{PV}, and Buchweitz and Van Straten~\cite[Main Theorem]{Bv}, have since given other proofs, in somewhat different contexts, of this result.  From \cite[Theorem 4.5]{MPSW2}, one easily deduces the following local version:

\begin{proposition}
\label{prop:MPSW-local}
Let $k$ be a perfect field and $Q=k[x_{1}, \dots, x_{n}]$ the polynomial ring with the standard grading. Let
$\ul{f}=f_{1},\dots, f_{c}$ be a $Q$-regular sequence of homogeneous polynomials, with $c\geq 2$.  Put $R= A_\fm$, where $\fm = (x_1,\dots,x_n)$.  Assume that  $A_\fp$ is a regular local ring for each $\fp$ in $\Spec(A)\!\setminus\!\{\fm\}$.  Then $\eta_c^R(M,N) = 0$ for all finitely generated $R$-modules $M$ and $N$. \end{proposition}

Very recently, M. E. Walker \cite{Wal} has eliminated the homogeneity requirement in characteristic zero:
\begin{theorem} (Walker, 2013) \label{thm:Walker} Let $k$ be a field of characteristic zero and $Q$ a smooth $k$-algebra. Let $\ul{f}=f_1, \dots, f_c$ be a $Q$-regular sequence,  with $c\geq 2$. Put $R = Q/(f_1, \dots, f_c)$. Assume the singular locus of $R$, $\{\fp \in \Spec(R): R_{\fp} \text{ is not regular} \}$, is a finite set of maximal ideals of $R$. Then  $\eta_c^R(M,N) = 0$ for all finitely generated $R$-modules $M$ and $N$.
\end{theorem}

 Of course $\eta_c$ vanishes over any localization of such a ring $R$ at a maximal ideal. This provides strong support for Conjecture~\ref{conj:MPSW2}, at least in characteristic zero.

\section{Vanishing of Ext, the depth formula, and torsion in $M\otimes_{R}M^{\ast}$}

\subsection*{Asymptotic vanishing} Over complete intersections there is an amazing connection between vanishing of $\Tor$ and vanishing of $\Ext$ \cite[Theorem 6.1]{AvBu}:

\begin{theorem}[Avramov -- Buchweitz, 2000] \label{AvBuThm} Let $(R,\fm)$ be a  complete intersection, and let $M$ and $N$ be finitely generated $R$-modules. Then $\Ext_{R}^{i}(M,N)=0$ for all $i\gg 0$ if and only if $\Tor_{i}^{R}(M,N)=0$ for all $i\gg 0$.
\end{theorem}

An obvious consequence of Theorem \ref{AvBuThm} is ``asymptotic Ext-symmetry'' over complete intersections: $\Ext_R^i(M,N) = 0$ for all $i\gg0$ if and only if $\Ext^i_R(N,M) = 0$ for all $i\gg0$.  A local ring with asymptotic Ext-symmetry  must be Gorenstein.  Avramov and Buchweitz \cite{AvBu} asked whether either (a) the class of  rings with asymptotic Ext-symmetry or (b) the class satisfying the conclusion of Theorem~\ref{AvBuThm} is \emph{strictly}  between the class of complete intersections and the class of Gorenstein rings. This motivated Huneke and Jorgensen to define and study ``AB'' rings in \cite{HuJ}.  (The name honors both Avramov-Buchweitz \cite{AvBu} and Auslander-Bridger \cite{AuBr}.)

\begin{definition}[Huneke and Jorgensen, 2003]\label{def:AB} An \emph{AB ring} is a local Gorenstein ring possessing a constant $C$, depending only on the ring and having the following property:  If $M$ and $N$ are finitely generated $R$-modules and $\Ext^i_R(M,N) = 0$ for all $i\gg0$, then $\Ext^i_R(M,N) = 0$ for all $i\ge C$.
\end{definition}

Avramov and Buchweitz \cite[4.7]{AvBu} proved that if $M$ and $N$ are finitely generated modules over a complete intersection  $R$, and if   $\Ext^i_R(M,N) = 0$ for all $i\gg0$, then $\Ext^i_R(M,N) = 0$ for all $i\geq \dim(R)+1$. Thus complete intersections are AB rings. Huneke and Jorgensen \cite[Theorem 4.1]{HuJ} showed that AB rings have aymptotic Ext-symmetry.  They  showed also that AB rings need not be complete intersections \cite[Theorem 3.5]{HuJ}  but left open  the question of whether or not every Gorenstein local ring is an AB ring.  This question was answered negatively by Jorgensen and \c Sega \cite{JS}.  We refer the reader to  \cite[Remark 4.5]{JS05} for a diagram
\cite[(4.15)]{JS05} of known inclusions and non-inclusions among these and other related classes of rings.  In particular, the class of AB rings and  the class of rings satisfying the conclusion of Theorem~\ref{AvBuThm}  both lie strictly  between the class of complete intersections and the class of Gorenstein rings.

\subsection*{The depth formula revisited} In another application of the AB property, Christensen and Jorgensen \cite[Corollary 6.4]{CJ} proved that the depth formula \eqref{eq:DF} holds for finitely generated modules over AB rings.  Since complete intersections are AB rings, this recovers the result of Huneke and Wiegand \cite[Proposition 2.5]{HW1}, stated above as Theorem~\ref{thm:DF}.

Working in the bounded derived category over a local ring $R$,
Foxby \cite{FoxbyJPAA}, \cite{Foxby} developed the notions of depth and projective dimension for complexes.  In \cite{Foxby} he
 relaxed the condition of Tor-independence and obtained a \emph{derived depth formula}  \cite[Lemma 2.1]{Foxby}; A special case of his result is the following: let $M$ and $N$ be $R$-modules (not necessarily finitely generated) and let $F_{\bullet}$ be a projective resolution of $M$. If $\pd_{R}(M)<\infty$, then $\depth(M)+\depth(N)=\depth(R)+\depth(F_{\bullet}\otimes N)$.  In case $\Tor^{R}_{i}(M,N)=0$ for all $i\geq 1$,  $\depth(F_{\bullet}\otimes_{R} N)=\depth(M\otimes_{R}N)$ so that Foxby's formula recovers Auslander's depth formula; see \cite{CJ}, \cite{Foxby}, and \cite{Iy} for details.

\subsection*{Torsion in $M\otimes_RM^*$}  Recall  \cite[Theorem 3.1]{HW1}, the main theorem of \cite{HW1}:

\begin{theorem}[Huneke and Wiegand, 1994] \label {HW1-main-thm} Let $M$ and $N$ be finitely generated modules over a hypersurface $R$, and assume that either $M$ or $N$ has rank.  If $M\otimes_RN$ is MCM, then both $M$ and $N$ are MCM, and either $M$ or $N$ is free.
\end{theorem}

The theorem breaks down for complete intersections: the ring $R=k[\![t^4,t^5,t^6]\!]$ is a complete intersection \cite{Her}, and the tensor product of the ideals $I = (t^4,t^5)$ and $J= (t^4,t^6)$ is torsion-free \cite[Example 4.3]{HW1}.  It seems much more difficult (in fact, impossible, so far!) to find such examples where $I$ and $J$ are non-principal ideals and $J \cong I^{-1}$.  More generally, one can ask:

\begin{conj}[Huneke and Wiegand, 1994]\label{conj:HW} Let $R$ be a local domain, and let $M$ be a finitely generated $R$-module.  If $M\otimes_RM^*$ is MCM, must $M$ be free?
\end{conj}

\noindent Here $M^*$ denotes the dual $\Hom_R(M,R)$.  Actually, Huneke and Wiegand \cite[pp. 473--474]{HW1} made this conjecture for torsion-free modules over one-dimensional Gorenstein domains, but both the special  and  the general cases are still wide open.

Motivated by this conjecture and its connection to the celebrated Auslander-Reiten Conjecture \cite{AR}, Celikbas and Takahashi \cite{CT} introduced the following condition on a local ring $(R,\fm)$:

\begin{quote}
\begin{enumerate}
\item[$\hwc$]
For every finitely generated  torsion-free $R$-module $M$, if $M\otimes_{R}M^{\ast}$ is reflexive, then $M$ is free.
\end{enumerate}
\end{quote}

The following example (see \cite[Example 3.5]{Da2} and \cite[Example 4.1]{HW1}) indicates why the tensor product is required to be reflexive, rather than just torsion-free:

\begin{example}\label{p0}
Let $R=k[[x,y,u,v]]/(xu-yv)$ and let $I=(x,y)R$. Then $R$ is a three-dimensional hypersurface with an isolated singularity and $I$ is MCM. Moreover $I^{\ast}\cong (x,v)R$ and hence $I\otimes_{R}I^{\ast}\cong (x,y,u,v)R$. Therefore $\depth(I\otimes_{R}I^{\ast})=1$, i.e., $I\otimes_{R}I^{\ast}$ is torsion-free, but $I\otimes_{R}I^{\ast}$ is not reflexive.
\end{example}

The \emph{Auslander-Reiten conjecture} \cite{AR}, transplanted to commutative algebra, asserts that every local ring satisfies the \emph{Auslander-Reiten condition}:
\begin{quote}
\begin{enumerate}
\item[$\arc$]
If $M$ is a finitely generated $R$-module such that $\Ext^{i}_{R}(M,M\oplus R)=0$ for all $i\geq 1$, then $M$ is free.
\end{enumerate}
\end{quote}
It is known that complete intersections satisfy $\arc$ \cite{ADS}.
Here is the connection between the Huneke-Wiegand conjecture and the commutative version of the Auslander-Reiten conjecture \cite[Theorem 1]{A}, \cite[Proposition 5.10]{CT}:

\begin{proposition}[Araya, 2009; Celikbas and Takahashi, 2011]\label{HWC-ARC}
Let $R$ be a local Gorenstein ring, and consider the following statements:
\begin{enumerate}[(i)]
\item
$R$ satisfies $\hwc$.
\item
$R_\fp$ satisfies $\hwc$ for every prime ideal $\fp$ with $\height \fp \le 1$.
\item
$R$ satisfies $\arc$.
\item $R_\fp$ satisfies $\arc$ for every prime ideal $\fp$ with $\height \fp \le 1$.
\end{enumerate}
Then $(ii) \Longrightarrow (i) \Longrightarrow (iii) \Longleftarrow (iv)$.
\end{proposition}
The proof  is short and is not explicitly given in \cite{CT}; hence we give it here.  The implication $(iii) \Longleftarrow (iv)$ is due to Araya \cite{A}. We will show that $(ii)\Longrightarrow(i)\Longrightarrow(iii)$. The key ideas, the first two used by
Auslander \cite[proof of Proposition 3.3]{Au}, are:
\begin{enumerate}[(a)]
\item The map $M\otimes_RM^* \to \Hom_R(M,M)$ taking $x\otimes f$ to the endomorphism
$y\mapsto f(y)x$ is an isomorphism if and only if $M$ is free.
\item A homomorphism $U \to V$, with $U$ reflexive and $V$ torsion-free, is an isomorphism if and only if it is an isomorphism when localized at all prime ideals of height at most one (See, e.g., \cite[Lemma 5.11]{LW}.)
\item The following theorem from \cite[Theorem 5.9]{HuJ}:
\end{enumerate}

\begin{theorem}[Huneke and Jorgensen, 2003]\label{thm:HuJ} Let $R$ be a $d$-dimensional local Gorenstein ring, and let $M$ and $N$ be MCM $R$-modules.  If $\Ext^i_R(M,N) = 0$ for all $i=1,\dots, d$, then $M\otimes_RN^*$ is MCM.
\end{theorem}

\begin{proof}[Proof of Proposition \ref{HWC-ARC}]
Set $d=\dim(R)$. Since $\Hom_R(M,M)$ is torsion-free whenever $M$ is torsion-free, the implication (ii)$\implies$(i) follows from (a) and (b).  Suppose now that (i) holds, and let $M$ be a finitely generated  $R$-module such that $\Ext^i_R(M,M\oplus R) = 0$ for all $i\ge 1$.  The vanishing of $\Ext_R^i(M,R)$ for all $i=1, \dots, d$ implies, by local duality \cite[3.5.7---3.5.9]{BH} that $M$ is MCM.   Now (c) and the fact that $\Ext^i_R(M,M)=0$ for all $i = 1,\dots, d$ implies that $M\otimes_RM^*$ is MCM. Since $R$ is Gorenstein, $M\otimes_RM^*$ is reflexive, and now (i) implies that $R$ is free.
\end{proof}

Notice that $\arc$ holds for the ring $k[\![x,y]\!]/(xy)$ (since it is a complete intersection), but $\hwc$ does not:  With $M = R/(x)$, we have $M\otimes_RM^* \cong M\otimes_RM \cong M$.  We do not know whether or not $\hwc$ and $\arc$ are equivalent for local Gorenstein \emph{domains}.

\section*{Acknowledgments}
The authors would like to thank Hailong Dao,  Srikanth Iyengar, and Mark Walker for their valuable comments during the preparation of this paper.  Thanks also to the referee for a careful reading and for several suggestions that have improved the paper.

Part of this work was completed when Celikbas visited the  University of Nebraska--Lincoln in October 2012 and January 2013. He is grateful for the kind hospitality of the UNL Department of Mathematics.

\bibliographystyle{plain}

\end{document}